\documentclass[reqno, 11pt]{amsart}
\usepackage{amsmath,amsfonts,amsthm,amscd,upref,amstext,enumerate}
\usepackage{amssymb,xspace}
\usepackage[dvips]{graphicx}
\usepackage{subfigure}
\usepackage[all]{xy}
\usepackage{comment}
\usepackage{xcolor}
\usepackage[colorinlistoftodos]{todonotes}

\newtheorem{theorem}{Theorem}[section]

\newtheorem{lem}[theorem]{Lemma}

\newtheorem{prop}[theorem]{Proposition}

\newtheorem{cor}[theorem]{Corollary}

\newtheorem{thm}[theorem]{Theorem}

\theoremstyle{definition}

\newtheorem{nota}[theorem]{Notation}

\newtheorem{example}[theorem]{Example}

\newtheorem{remark}[theorem]{Remark}

\newtheorem{Setting}[theorem]{Setting}

\numberwithin{equation}{section}

\newtheorem{remark/questions}[theorem]{Remark and Questions}
\newtheorem{fact}[theorem]{Fact}

\newcommand{\N}{\mathbb{N}}

 \long\def\alert#1{\smallskip\line{\hskip\parindent\vrule
\vbox{\advance\hsize-2\parindent\hrule\smallskip\parindent.4\parindent
  \narrower\noindent#1\smallskip\hrule}\vrule\hfill}\smallskip}

\newtheorem{abf}[theorem]{Auslander-Buchsbaum Formula}

\newcommand{\m}{\mathfrak{m}}
\newcommand{\n}{\mathfrak{n}}

\DeclareMathOperator{\pd}{pd}

\DeclareMathOperator{\Tor}{Tor}
\DeclareMathOperator{\tor}{Tor}

\DeclareMathOperator{\depth}{depth}

\def\dim{\mathop{\rm dim}}
 \def \N{{\mathbb  N}}

\numberwithin{equation}{section}

\begin{document}


\title{Vanishing of Tor over fiber products}



\author{T. H. Freitas}
\address{Universidade Tecnol\'ogica Federal do Paran\'a, 85053-525, Guarapuava-PR, Brazil}
\email{freitas.thf@gmail.com}

\author{V. H. Jorge P\'erez}
\address{Universidade de S{\~a}o Paulo -
ICMC, Caixa Postal 668, 13560-970, S{\~a}o Carlos-SP, Brazil}
\email{vhjperez@icmc.usp.br}

\author{R. Wiegand}
\address{University of Nebraska-Lincoln}
\email{rwiegand@unl.edu }

\author{S. Wiegand}
\address{University of Nebraska-Lincoln}
\email{swiegand1@unl.edu}

\date{August 27, 2019}
\thanks{All four authors were partially supported by FAPESP-Brazil  2018/05271-6,   
2018/05268-5 and CNPq-Brazil 421440/2016-3.
RW was partially supported by Simons Collaboration Grant 426885. SW was partially 
supported by a UNL Emeriti \&
Retiree Association Wisherd Award}


\begin{abstract}
Let $(S,\m,k)$ and $(T,\n,k)$ be local rings, and let $R$ denote their fiber
 product over their common residue field $k$.  Inspired by work of Naseh and Sather-Wagstaff, we explore consequences of
 vanishing of $\Tor^R_m(M,N)$ for various values of $m$, where $M$ and $N$ are
  finitely generated $R$-modules.
\end{abstract}

\maketitle

\section{Introduction}  Recently there has been renewed interest in the homological properties of 
 fiber product rings.
In particular, the results obtained by Nasseh and Sather-Wagstaff on the vanishing of Tor  in \cite{NS} have  inspired
us to  try to extend  their computations.   This note should be regarded as an addendum to that paper, or perhaps an advertisement for the utility of the nice results established there.

\begin{Setting} \label{fpset} Let $(S,\mathfrak{m},k)$ and $(T,\mathfrak{n},k)$
be commutative local rings.
Let $S \stackrel{\pi_S}\twoheadrightarrow k \stackrel{\pi_T}\twoheadleftarrow T$
 denote the natural surjections onto the common residue field,
and assume that $S\neq k\neq T$. Let $R$ denote the fiber product:

$$
R:= S \times_k T=\{(s,t)\in S\times T  \ \mid \ \pi_S(s)=\pi_T(t)  \}.
$$
Then $R$ is a local ring with maximal 
ideal $\m\times\n$, and $R$ is a
subring of the usual direct product $S\times T$. Let $\eta_S:R\twoheadrightarrow S$
and $\eta_T:R\twoheadrightarrow T$ be the
projections $(s,t)\mapsto s$ and $(s,t)\mapsto t$, respectively.
The maps $\eta_S$ and $\eta_T$ are surjective, with respective
kernels $J:= 0\times\n$ and $I:= \m\times0$.
Then $R$ is represented as a pullback diagram:

\begin{equation*}
\begin{CD}
R & \stackrel{\eta_S}{\longrightarrow} & S\\
{\eta_T}{\downarrow} &  & {\downarrow}{\pi_S} \\
T & \stackrel{\pi_T}{\longrightarrow} & k
\end{CD}
\tag{\ref{fpset}.1}
\end{equation*}

\
\

\noindent The maximal ideal $\m\times \n$ is decomposable: $\m\times \n = I\oplus J$.
For future reference we note that
\begin{equation*}\label{eq:iso}
\begin{gathered}
I\cong \m \quad\text{and}\quad J\cong\n \quad\text{as
$R$-modules}\,;  \quad\text{also}\\
\ S\cong R/J \quad\text{and}\quad T\cong R/I\,\quad\text{as rings and as
$R$-modules.}\,
\end{gathered}\tag{\ref{fpset}.2}
\end{equation*}

Whenever we use the symbols $R$, $S$, or $T$, we tacitly assume that $R$ is the fiber product
as described here; the notation $\m$, $\n$, $I$, and $J$ for the various ideals will be preserved
throughout the paper.
Moreover, every module over $R$, $S$, or $T$ is assumed to be finitely generated.
\end{Setting}

        The general theme of the paper is to assume the vanishing of $\Tor_m^R(M,N)$ for
        certain $R$-modules $M$ and $N$, and
        certain values of $m$, and then describe properties of the modules that result from
        this assumption.

 \begin{nota} \label{syzbnot}    For a local ring $(A,\m_A)$ and a finitely generated
$A$-module $M$, we let $\Omega_A^iM$ denote the $i^{\text{th}}$ syzygy of
$M$ with respect to a minimal $A$-free resolution. We often write $\Omega_AM$ for $\Omega^1_AM$.  
The $i^{\text{th}}$ Betti number $\beta^A_iM$ is the minimal number of generators 
required for the $A$-module $\Omega^i_AM$.

For an  $A$-module $L$ and a non-negative integer $c$, we denote by $L^{\oplus c}$ the direct sum of\ \ $c$\  \ copies of $L$.  
(When $L$ is an ideal of a ring $A$, this distinguishes the direct sum from the ideal product $L^c = L\cdots L$.)
\end{nota}
First we focus on  vanishing of $\Tor^R_m(M,N)$, where $M$ and $N$ are $S$-modules or $T$-modules, using  two useful lemmas. Lemma~\ref{lem:syz} is due to  Nasseh and Takahashi, and Lemma~\ref{lem:tor1} is from the paper \cite{NS} by Nasseh and Sather-Wagstaff and consists of the first and sixth of eight formulas 
given in \cite[Lemma 2.3]{NS}.
  \medskip

        \begin{lem}\label{lem:syz} $\phantom{x}$ \cite[Lemma 3.2]{NT} Let $Y$ be an $S$-module, and set $b=\beta_0^SY$.  Then  
        $$\Omega_RY \cong J^{\oplus b}\oplus \Omega_SY\cong \n^{\oplus b}\oplus \Omega_SY.$$
        \end{lem}

\begin{lem}\label{lem:tor1} \cite[Lemma 2.4]{NS}. Let $X$ and $Y$  be  $S$-modules, and let $Z$ be  a $T$-module. 
\begin{enumerate}
\item 
$\Tor^R_1(X,Y) \cong \Tor_1^S(X,Y)\oplus \left(\frac{Y}{\m Y}\right)^{\oplus(\beta_1^Tk)(\beta_0^SX)}$. 
\item
$\Tor^R_1(Y,Z) \cong \Tor_1^S(Y,k)^{\oplus\beta_0^TZ}\oplus \left(\frac{Y}{\m Y}\right)^{\oplus\beta_1^TZ}$.
\end{enumerate}
\end{lem}

    \begin{prop}\label{prop:tor1} Let $Y$ be a non-zero $S$-module  and $Z$ a  $T$-module. 
    \begin{enumerate}
\item If $\Tor_1^R(Y,Z) = 0$, then $Z$ is a free $T$-module. 
\item If $Y$ and $Z$ are free over $S$ and $T$, respectively, then $\Tor_1^R(Y,Z) = 0$.
\end{enumerate}
\end{prop} 
 \begin{proof}
If $\Tor_1^R(Y,Z) = 0$, then $\big(\frac{Y}{\m Y}\big)^{\oplus\beta_1^TZ}=0$, by Lemma \ref{lem:tor1}.
    Since $\frac{Y}{\m Y} \ne 0$ by Nakayama's Lemma, we have $\beta^T_1Z=0$, that is, $Z$ is free as a $T$-module.  For the converse, we may assume $Y=S$ and $Z=T$.  As in \cite[Remark 2.6]{NS}, 
    one has $\Tor_1^R(S,T) =\Tor_1^R(R/J,R/I)\cong \frac{J\cap I}{JI}  = 0$. 
        \end{proof}

When $Y$ is a free $S$-module, we obtain quantitative information about $\Tor_1^R(Y,Z)$:

    \begin{prop}\label{prop:torSZ} Let $Z$ be a $T$-module and $a\in \N$.  Then $\Tor^R_1(S^{\oplus a},Z)$ 
    is a $k$-vector space whose dimension is equal to
    $a\beta_1^TZ$.
        \end{prop}
    \begin{proof}  Letting $n=\beta^T_0Z$ and  $Z_1 = \Omega^1_TZ$, we get an exact sequence
    $$
    0 \to Z_1 \to T^n \to Z \to 0.
    $$
   For $a=1$, apply the functor $S\otimes_R - \ \ \big(= (R/J)\otimes_R - \big) $  to this short exact sequence
   to get the exact 
 sequence
    $$
    0 \to \Tor_1^R(S,Z) \to Z_1/JZ_1 \to (T/J)^n \to Z/JZ \to 0 \,.
    $$
    The zero on the left is by Proposition~\ref{prop:tor1}(2).  Each of these modules is annihilated by
    the maximal ideal $I\oplus J$  of $R$, so they are $k$-vector spaces.  Moreover the last two non-zero terms have the same $k$-dimension, namely $n$.  It follows that the first two terms have the same dimension, and hence that $\dim_k\big(\Tor_1^R(S,Z)\big) =
    \dim_k(Z_1/JZ_1) = \dim_k(Z_1/\n Z_1) = \beta_0^TZ_1 = \beta_1^TZ$.
        
        For $a>1$, $\Tor_1^R(S^{\oplus a},Z) =  \Tor_1^R(S,Z)^{\oplus a}$ and the dimension is  $a\beta_1^TZ$.
        \end{proof}

 Part (1) of Theorem~\ref{evodthm} is      a generalization of a result of  Nasseh and Sather-Wagstaff
        \cite[Lemma 2.4]{NS} and part (2) generalizes Proposition~\ref{prop:tor1}(1).

\begin{theorem} \label{evodthm} Let $X$  and $Y$ be $S$--modules,  let $Z$ be a
$T$-module, 
and let  $m$  be a 
nonnegative integer. 
 \begin{enumerate} 
\item[\rm (1)]$\Tor^R_{2m+1}(X,Y)= 0 \implies X=0$ or $Y=0$. 
   \item[\rm (2)]$\Tor^R_{2m}(Y, Z)= 0 \implies Y=0$ or $Z=0$.
   \item[\rm (3)]$\Tor^R_{2m}(X,Y)= 0 $  and $X\ne 0 \implies Y$ is a free $S$-module. 
   \item[\rm (4)] If $T$ is not a discrete valuation domain (DVR),
    then \\ $\Tor^R_{2m}(X,Y)= 0 \implies X=0$ or $Y=0$. 
   \item[\rm (5)]$\Tor^R_{2m+1}(Y, Z)= 0$  and $Y\ne 0 \implies Z$  is a free $T$-module. 
   \item[\rm (6)] If $S$ is not a DVR and $m> 0$,
   then $\Tor^R_{2m+1}(Y, Z)= 0 \implies Y=0$ or $Z=0$.
\end{enumerate}
By symmetry, the corresponding statements hold if $X$  and $Y$  are $T$--modules and $Z$ is an
$S$-module. 
\end{theorem}
\begin{proof} 
  Put $a=\beta_0^SX$,   $b=\beta_0^SY$, $c = \beta_0^TZ$, and  $d=\beta_1^Tk= \beta^T_0\n$.  The assumption $T\ne k$ in
  Setting \ref{fpset} implies that $d\ne 0$.

The case $m=0$ of  (1) is   
  \cite[Lemma 2.4]{NS}; we give their proof for completeness.    
  Thus we assume $\Tor^R_{1}(X,Y) = 0$ and $X\ne 0$ (so $a\ne 0$).  
 By Lemma~\ref{lem:tor1}(1), 
  \[
 \Tor^R_{1}(X,Y) \cong \Tor^S_{1}(X,Y)\oplus 
\big(\frac{Y}{\m Y}\big)^{\oplus da}\,.
\]
Since $da\ne0$, we have $\frac{Y}{\m Y} = 0$. Hence $Y=0$, and so (1) holds for $m=0$.
 
 For (1) with $m>0$, proceed by induction.  Assume 
$\Tor^R_{2m+1}(X,Y) = 0$,  
and that $$\Tor^R_{2m-1}(X',Y')=0\implies X'=0\text{ or }Y'=0$$
whenever $X'$ and $Y'$ are $S$-modules.  By symmetry,  
 $$\Tor^R_{2m-1}(V,W)=0\implies V=0\text{ or }W=0$$   
whenever $V$ and $W$ are $T$-modules.  
  Using Lemma~\ref{lem:syz}, we get
$$\aligned
0&=\Tor^R_{2m+1}(X,Y)=\Tor^R_{2m}(\Omega_RX,Y)=\Tor^R_{2m-1}(\Omega_RX,\Omega_RY)\\
&=\Tor^R_{2m- 1}(\n^{\oplus a}\oplus \Omega_SX,  
\n^{\oplus b}\oplus \Omega_SY)=\Tor^R_{2m-1}(\n^{\oplus a},  \n^{\oplus b}) \oplus \boxed{\text{other terms}}.
\endaligned$$
This implies $0=\Tor^R_{2m-1}(\n^{\oplus a},  \n^{\oplus b})$. 
Since $\n^{\oplus a}$ and $\n^{\oplus b}$ are both $T$-modules, the inductive hypothesis implies that
 $\n^{\oplus a} = 0$ or $\n^{\oplus b}=0$. 
But $\n\ne 0$, and so  $a=0$ or $b=0,$ that is, $X=0$ or $Y=0$, as desired for (1). 


If $Y\otimes_RZ=0$, then $\overline Y\otimes_k\overline Z = 0$, where $\overline Y$ and $\overline Z$ are reductions modulo the maximal ideal of $R$.  Therefore either $\overline Y=0$ or $\overline Z = 0$, and Nakayama's Lemma implies that $Y=0$ or $Z=0$.  This gives the case $m=0$ in (2), and the same argument
verifies the case $m=0$ in statements   (3)  and (4).

To see that (2) holds for  $m\ge 1$, we use Lemma~\ref{lem:syz}: 
$$
\aligned 0&=\Tor^R_{2m}(Y,Z)=  \Tor^R_{2m-1}(\Omega_R Y,Z) = 
\Tor^R_{2m-1}(\n^{\oplus b}\oplus \Omega_SY,Z)\\
&=\Tor^R_{2m-1}(\n^{\oplus b},Z)
\oplus \Tor^R_{2m-1}( \Omega_SY,Z)
 \implies 0
=\Tor^R_{2m-1}(\n^{\oplus b},Z).\endaligned
$$ 
By  Equation~\eqref{eq:iso}, $\n^{\oplus b}$ is a $T$-module,
as is $Z$. If $Z\ne 0$, then part (1) 
(applied to $T$) shows $\n^{\oplus b} = 0$.  One of the assumptions in 
Setting~\ref{fpset} is that $\n\ne0$.  Therefore $\beta^S_0Y = b = 0$, that is, $Y=0$. A symmetric 
argument shows that $Y\ne 0$ implies $Z=0$, and so (2) holds.

Next observe that Proposition~\ref{prop:tor1}(1) implies that (5)  holds for $m=0$.  Taking first syzygies
over $R$, shifting, and using Lemma~\ref{lem:syz} (and its counterpart over $T$), we obtain, for every $s\ge 2$,
\begin{equation}\label{eq:syz-shift}
\begin{gathered}
\Tor_s^R(X,Y) =  \Tor^R_{s-1}(X,\n^{\oplus b}) \oplus \Tor_{s-1}^R(X,\Omega_SY) \quad\text{and} \\
\Tor_s^R(Y,Z) =  \Tor^R_{s-1}(Y,\m^{\oplus c}) \oplus \Tor_{s-1}^R(Y,\Omega_TZ) \,.
\end{gathered}
\end{equation}

Setting $s=2$ in the first equation, we deduce the case $m=1$ of (3) from the case $m=0$ of (5).  Setting $s=3$ in the second equation, we deduce the case $m=1$ of (5) from the case $m=1$ of (3).  Next we set $s=4$ in the first equation and get the case $m=2$ of (3) from the case $m=1$ of (5).  Continuing in this way, alternating between the two equations, we obtain (3) and (5) by induction.  

For (4), suppose $m>0$ , $X\ne 0$, and $\Tor^R_{2m}(X,Y) = 0$.  Using Lemma~\ref{lem:syz}, we learn that
  $\Tor_{2m-1}^R(X, \n^{\oplus b}) = 0$.  Now (5) implies that $\n^{\oplus b}$ is a free $T$-module.  Since $\n$ is {\em not} $T$-free (as $T$ is not a DVR) we conclude that $b=0$, that is, $Y=0$.

To prove (6), suppose $m>0$, $Y\ne0$, and $\Tor^R_{2m+1}(Y,Z) = 0$.  Then 
$\Tor^R_{2m}(Y,\m^{\oplus c})=0$, and now (3) implies that $\m^{\oplus c}$ is free.  As in the proof of (4), it
follows that $Z=0$.
\end{proof}
We note that item (6) of Theorem~\ref{evodthm} does not hold when $m= 0$, since one always has
$\Tor^R_1(S,T) = 0$.

\begin{cor} \label{thm:tor3}
Let  $Y$ and $Z$ be non-zero modules over $S$ and $T$ respectively, and let  $m\ge 0$.
If    $\Tor_m^R(Y,Z) = 0$, then $m$ is odd and
    $Y$ and $Z$ are free.  \end{cor}
 \begin{proof}
Use (2) and (5) of Theorem~\ref{evodthm}.
\end{proof}

\begin{remark}\label{rem:depth}
$\depth R = \min\{\depth S, \depth T, 1\}\,.$
\end{remark}
Remark~\ref{rem:depth} follows from the work of Lescot \cite{L81}.  See, for example, \cite[(3.2) Remark]{CSV}.
Actually, a low-tech, direct proof is easy:  Note first that an element $(s,t)\in R$ is a non-zerodivisor (NZD) of $R$ if and only if $s$ is a NZD of $S$ and $t$ is a NZD of $T$.  It follows that $\depth R >0 \iff \depth S >0$ and $\depth T>0$.  To see that $\depth R \le 1$, suppose that $(s,t)$ is a NZD of $R$ in $\m\times\n$, and let $(u,v)$ be a arbitrary element of $\m\times \n$.  Then $(s,0)(u,v) = (u,0)(s,t)\in R(s,t)$. Moreover, $(s,0)\notin R(s,t)$:
Indeed, if $(s,0)= (a,b)(s,t)$, then $b=0$, as $t$ is a NZD; also, the equation $as=s$ forces $a=1$, a contradiction, since $(1,0)\notin R$.  Thus every element of $\m\times \n$ is a zero-divisor modulo $R(s,t)$.

\begin{abf}\label{rem:ABF} 
\cite[A.5. Theorem, p. 310]{LW} Let $M$ be a finitely generated module  of finite projective dimension over a local ring $(A,\m_A)$. Then $$\depth M+\pd_AM=\depth A.$$
\end{abf}

From now on, our conclusions are going to be that one of the modules has finite projective dimension over $R$.  It is important to realize, however, the following fact:

\begin{fact}\label{pdfin1} If  $M$ is an
 $R$-module with $\pd_RM <\infty$, then  $\pd_RM \le 1$.
 \end{fact}

 \begin{proof} See the Auslander-Buchsbaum Formula~\ref{rem:ABF} 
and Remark~\ref{rem:depth}.
\end{proof}

We now analyze implications of $\Tor_\m^R(M, Y)=0$, where $M$ is an 
arbitrary $R$-module, that is, not necessarily  an $S$-module or a $T$-module.  The conditions imposed in the next theorem may appear a bit contrived, but Example~\ref{eg:DVRs}, which follows the proof of the theorem, shows that they are exactly what is needed.

\begin{thm}\label{thm:tor4} Let $M$ be an $R$-module and $Y$ a non-zero $S$-module.  Assume
that $\Tor^R_m(Y,M) = 0$ for some $m\ge4$ and that 
at least one of these two
 conditions holds:
\begin{enumerate}
\item $T$ is not a discrete valuation ring, or
\item $Y$ is not a free $S$-module.
\end{enumerate}
Then $\pd_RM \le 1$.

If, in addition,  $S$ or $T$ has depth $0$, then $M$ is a free
$R$-module.
\end{thm}

The proof uses a lemma due to Dress and Kr\"amer.

\begin{lem}\label{lem:Omega2splits}$\phantom{x}$\cite[\emph{Bemerkung} 3]{DK75} Let $M$ be an $R$-module.  Then
$\Omega^2_RM$ decomposes as a direct sum:  $\Omega^2_RM = X\oplus Z$, where $X$
is an $S$-module and $Z$ is a $T$-module.
\end{lem}

\begin{proof}[Proof of Theorem~\ref{thm:tor4}]
Set $N=\Omega_R^{m-4}M$.   If we  show that $\pd_RN<\infty$, it
will follow that $M$ too has finite projective dimension, and hence projective dimension at most $1$.  Moreover, if either $S$ or $T$ has depth $0$, then $\depth R = 0$ by Remark~\ref{rem:depth}, 
and hence $\pd_RM\le 0$ by Remark~\ref{rem:ABF}.   Thus our goal is to show that $N$ has finite projective dimension if $\Tor_4^R(N,Y) = 0$.  

Write
$\Omega^2_RN\cong X \oplus Z$ as in Lemma~\ref{lem:Omega2splits}, where  $X$
is an $S$-module and $Z$ is a $T$-module.  Since  $\Tor_2^R(Y, \Omega_R^2N) = \Tor^R_4(Y,N)=0$, 
we have
\[
\Tor_2^R(Y,X) = 0 = \Tor_2^R(Y,Z)\,.
\]
Since $Y\ne 0$, Theorem~\ref{evodthm}(2) 
implies that $Z=0$. If also $X=0$, then  $\Omega^2_RN =0$, and so
$\pd N_R\le1$, and we are done. 

If, on the other hand, $X\ne0$, then Theorem~\ref{evodthm}(3) shows
that $Y$ is a free $S$-module.
 Therefore assumption (2) fails, so (1) must hold, that is,
 $T$ is not a discrete valuation ring. Now  Theorem~\ref{evodthm}(4) yields $Y= 0$,  a contradiction.
 \end{proof}

Nasseh and Sather-Wagstaff ask \cite[Question 2.14]{NS} whether the
vanishing of $\Tor^R_4(M,N)$ (for a fiber product $R$) forces one of
the modules to have finite projective dimension. The following
example shows that the answer is ``no'' and justifies the hypotheses
imposed in Theorem~\ref{thm:tor4}.  (But see Corollary~\ref{cor:torsionless} for 
a result in the positive direction.)  The example also shows the need
for {\em two} vanishing Tors in the hypotheses of
\cite[Theorem 1.1(b)]{NS}. 

\begin{example}\label{eg:DVRs} Let $(S,\m,k)$ and $(T,\n,k)$ be
discrete valuation rings, and let $R$ be the fiber product of $S$ and $T$.
Then  $\Omega_RS= \Omega_R(R/J) = J \cong \n \cong T$, since $\n$ is a
 principal ideal in the domain $T$.  Similarly $\Omega_RT=S$.
 Both $S$ and $T$ have non-zero annihilators and therefore are not
 free as $R$-modules.  It follows, from the syzygy relations above, that both $S$ and $T$
have infinite projective dimension over $R$.
However,

$$\Tor^R_2(S,S) =\Tor^R_1(S,\Omega_RS) \cong\Tor^R_1(S,T) = 0,$$
by Proposition~\ref{prop:tor1}(2). Similarly 

$$\Tor^R_4(S,S) =\Tor^R_2(\Omega_RS,\Omega_RS) \cong\Tor^R_2(T,T)= 0.$$
It follows that
$\Tor^R_m(S,S) = 0$ and $\Tor^R_m(T,T) = 0$ for every even positive integer,
and $\Tor^R_m(S,T) = 0$ for every odd positive integer.
\end{example}

\begin{thm}\label{thm:tor5} Let $M$ and $N$ be $R$-modules with
$\Tor_5^R(M,N) = 0$.  Then at least one of the following four things happens:
\begin{enumerate}
\item $\pd_RM\le 1$.
\item $\pd_RN\le 1$.
\item $\Omega^2_RM$ is a free $S$-module and $\Omega^2_RN$ is a free $T$-module.
\item $\Omega^2_RM$ is a free $T$-module and $\Omega^2_RN$ is a free $S$-module.
\end{enumerate}
\end{thm}
\begin{proof}  Using Lemma~\ref{lem:Omega2splits}, write $\Omega^2_RM = X\oplus Z$ and
$\Omega^2_RN = Y\oplus W$, where $X$ and $Y$ are $S$-modules and $Z$ and $W$ are $T$-modules.
Now 
$$\Tor^R_1(X\oplus Z,Y\oplus W)=\Tor^R_1(\Omega^2_RM,\Omega^2_RN) = \Tor^R_5(M,N) = 0.$$ Hence
\[
\Tor^R_1(X,Y) = \Tor^R_1(X,W)=\Tor^R_1(Z,Y) = \Tor^R_1(Z,W)=0\,.
\]
From Theorem~\ref{evodthm}(1), $\Tor^R_1(X,Y) =0$ and $\Tor^R_1(Z,W)=0$ imply
 \[
\begin{gathered}
X=0 \quad\text{or}\quad Y=0\,;\qquad\text{and}\\
Z=0\quad\text{or}\quad W=0\,.
\end{gathered}
\]

If $X=Z=0$ we get (1), and if $Y=W = 0$ we get (2).
There are two remaining cases:

(a) $Y=0=Z$ and $X\ne 0 \ne W$\,.

(b) $X=0 = W$ and $Y\ne0 \ne Z$\,.

\noindent Assume (a) holds.  Then $\Omega^2_RM= X$, which is an $S$-module, and $\Omega^2_R N = W$, a $T$-module.  Now Theorem~\ref{evodthm} and the equation $\Tor_1^R(X,W) = 0$ yield conclusion (3).  Similarly, case (b) leads to conclusion (4).
\end{proof}

Now on to $\Tor_6$\,. Here, as in the result above,
both $M$ and $N$ are allowed to be arbitrary $R$-modules (that is, not
necessarily annihilated by $I$ or by~$J$).
In view of Example~\ref{eg:DVRs}, however, we cannot do away with the assumption that 
$S$ and $T$ are not DVRs.

\begin{thm}\label{thm:tor6} Let $M$ and $N$ be $R$-modules. Assume
that
 neither $S$ nor $T$ is a discrete valuation domain.  If $\Tor_m^R(M,N)= 0$ for some $m\ge 6$, then
$\pd_RM\le 1$ or $\pd_RN\le 1$.  Therefore $\Tor_i^R(M,N) = 0$ for all $i\ge2$.
\end{thm}
\begin{proof} If $m>6$ we easily reduce to the case $m=6$ by
taking syzygies of one of the modules.  Therefore we assume that
$m=6$. 
Assume $\pd_RN>1$; we  show that $\pd_RM\le
1$. From Lemma~\ref{lem:Omega2splits} we have
$\Omega^2_RN\cong Y\oplus W$, where $Y$ is an $S$-module and
$W$ is a $T$-module.  Moreover, since $\pd_RN >1$, at least one of $Y$ and $W$ is non-zero.
From the equation $\Tor^R_4(M,\Omega_RN) = 0$, we get
\[
\Tor^R_4(M,Y) = 0 \quad\text{and}\quad \Tor^R_4(M,W) = 0\,.
\]

If $Y \neq 0$,  Theorem~\ref{thm:tor4} and our assumption that $T$ is not a discrete valuation ring imply  that  $\pd_RM\le 1$.  If $Y=0$, then $W\ne 0$, and the assumption that $S$ is not a DVR shows that $\pd_RM\le 1$ in this case.
\end{proof}

\begin{thm}\label{thm:tor6s} Let $M$ and $N$ be $R$-modules such that $\Tor_m^R(M,N)= 0$ for some ~$m\ge 6$. Using Lemma~\ref{lem:Omega2splits}, write 
$$
\Omega^2_RM\cong X\oplus Z \quad \text{ and } \quad \Omega^2_RN\cong Y\oplus W\,,
$$ 
where $Y$ and $X$ are $S$-modules and 
$W$ and $Z$ are $T$-modules. Assume
that at least one of the following conditions holds:
 \begin{enumerate}
 \item[\rm (1)] At least one of $X$ and  $Y$ is not free as an $S$-module; or 
\item[\rm (2)] At least one of $Z$ and $W$  is not free as a $T$-module; or
\item[\rm (3)] At least one of $\{X,Y\}$ is $0$ AND at least one of $\{Z,W\}$ is $0$; or
\item [\rm (4)]Neither $S$ nor $T$ is a discrete valuation domain.
\end{enumerate}
  Then
~$\pd_RM\le 1$ or $\pd_RN\le 1$.  Therefore $\Tor_i^R(M,N) = 0$ for all $i\ge2$.
\end{thm}
\begin{proof}If (4) holds, we quote Theorem~\ref{thm:tor6}.  For the other cases, 
we may assume, by taking syzygies, that
$m=6$. 
We then have 
$$
\Tor^R_4(M,\Omega^2_RN) = 0 = \Tor^R_4(\Omega_R^2M,N)\,,
$$
 and hence
\begin{equation}\label{eq:4}
\aligned
\Tor^R_4(M,Y)& = 0  \quad\text{and}\quad \Tor^R_4(M,W) = 0\,;\\
 \Tor^R_4(X,N)& = 0  \quad\text{and}\quad \Tor^R_4(Z,N) = 0\,.
 \endaligned
\end{equation}
If $Y$ is not $S$-free,  the first equation above and Theorem~\ref{thm:tor4} show that $\pd_RM\le1$.
Symmetric arguments give the desired conclusion under either assumption (1) or (2).  The remaining case to consider is (3).

If $X=0=Z$, then $\Omega^2_RM= 0$, and hence $\pd_RM\le 1$.  Similarly, if $Y=0=W $, then
$\pd_RN\le 1$.  The remaining possibilities, under assumption (3), are (a) $X=0=W$ and (b) $Y=0=Z$.  By 
symmetry, we need to consider only possibility (a) $X=0=W$.  We have, from the first equation in \eqref{eq:4}, 
$$ 
\Tor_2^R(X\oplus Z,Y) =  \Tor_2^R(\Omega^2_RM,Y) = 0\,, \quad\text{and hence}\quad
\Tor^R_2(Z,Y) = 0\,.
$$ 
 By Theorem~\ref{evodthm}(2), either $Y=0$, in which case
$\Omega^2_RN=0$; or $Z=0$, in which case $\Omega^2_RM=0$.  Thus either $\pd_RN\le1$ or
$\pd_RM\le 1$, as desired.
\end{proof}

\bigskip



 The next theorem is a slight
 generalization of a result due to Nasseh and Sather-Wagstaff \cite[Theorem 1.1 (b)]{NS}. 

\begin{thm}\label{evenoddtors} Let $M$ and $N$ be finitely generated modules over a fiber product $R$ with
$\Tor^R_{2i+1}(M,N) = 0= \Tor_{2j}^R(M,N),$ for some $i\geq 2$ and $j\geq 3$.  Then 
$\pd_RM\le 1$ or  $\pd_RN\le 1$.
\end{thm}
\begin{proof} We consider separately the cases where $2i+1<2j$ and $2i+1>2j$.

Case 1: First suppose $2i+1<2j$.  If $i>2$, then  
$$
\Tor_{5}^R(\Omega_R^{2i+1-5}M,N) = \Tor_{2i+1}^R(M,N) = 0\,.
$$ 
Since it suffices to show either $\pd (\Omega_R^{2i-4}M)$ or $\pd N$ is finite, we may suppose that $i=2$; that is, 
$$
\Tor_{5}^R(M,N) = 0= \Tor_{2j}^R(M,N)\,,
$$ 
where $j\ge 3$. 

By Theorem~\ref{thm:tor5}, at least one of the following four things happens:
\begin{enumerate}
\item $\pd_RM\le 1$.
\item $\pd_RN\le 1$.
\item $\Omega^2_RM$ is a free $S$-module and $\Omega^2_RN$ is a free $T$-module.
\item $\Omega^2_RM$ is a free $T$-module and $\Omega^2_RN$ is a free $S$-module.
\end{enumerate}
If either of the first two happens, we are done. If not, suppose (3) happens, so that  $\Omega^2_RM=S^{\oplus a}$  
and $\Omega^2_RN=T^{\oplus b}$.

Suppose $a\ne 0 \ne b$. Now $0= \Tor_{2j}^R(M,N)=\Tor^R_{2j-4}(S^{\oplus a}, T^{\oplus b})$ implies 
$$
0=\Tor^R_{2j-4}(S, T)=\Tor^R_{2j-5}(S, \Omega_RT)=\Tor^R_{2j-5}(S, \m)\,,
$$
the last equality from \eqref{eq:iso}.  Since both $S$ and $\m$ are non-zero $S$-modules, 
this contradicts Theorem~\ref{evodthm}(1).  
Therefore either $\Omega^2_RM=S^{\oplus a}=0$ or $\Omega^2_RN=T^{\oplus b}=0$. Thus  $\pd_RM\le 1$ or  $\pd_RN\le 1$.

A similar argument works if (4) happens.

Case 2: Assume that  $2i+1>2j$.  By taking syzygies, we can reduce to the case where $j=3$, and so $\Tor_6^R(M,N)=0$. As in Theorem~\ref{thm:tor6s}, let $\Omega^2_RM=X\oplus Z$  and $\Omega^2_RN=Y\oplus W$. We may take  each of the pieces to be free, 
so that $X= S^{\oplus a}, Y= S^{\oplus b}, Z=T^{\oplus c}, W=T^{\oplus d}.$

Now we also  have  $0=\Tor_{2i+1}^R(M,N) = \Tor_{2i-3}^R(\Omega_R^{2}M,\Omega_R^{2}N)$. 
Thus 
\begin{equation*}\aligned0=&\Tor_{2i-3}^R(S^{\oplus a}\oplus T^{\oplus c},S^{\oplus b}\oplus T^{\oplus d})\implies \\
&\Tor_{2i-3}^R(S^{\oplus a},S^{\oplus b})=0=\Tor_{2i-3}^R(T^{\oplus c}, T^{\oplus d}).
\endaligned
\end{equation*}
By Theorem~\ref{evodthm}(1), since $2i-3$ is an odd positive number, we have $a=0$ or $b=0$; AND $c=0$ or $d=0$. 
Thus condition (3) of Theorem~\ref{thm:tor6s} holds, and so 
$\pd_RM\le 1$ or
 $\pd_RN\le 1$.\end{proof}

Recall that a finitely generated module $M$ over a Noetherian ring
$A$ is  {\em torsionless} \cite{Ba} provided the canonical
biduality map $\delta_M:M\to M^{**}$ is injective.  By mapping a finitely generated free module onto $M^*$
and then dualizing, we get an embedding of $M^{**}$ into a free module.  It follows that every torsionless module over a local ring is, up to free
summands, a syzygy module. Therefore we get the following corollary of
Theorem~\ref{thm:tor6} by representing each of the two modules as a syzygy and
then shifting up two homological degrees.

\begin{cor}\label{cor:torsionless} Assume that neither $S$ nor $T$ is a DVR, and
let $M$ and $N$ be torsionless $R$-modules. If
$\Tor_4^R(M,N) = 0$, then at least one of $M$ and $N$ has projective dimension at most one.  \qed
\end{cor}

\end{document}